\documentclass[12pt,reqno]{amsart}
\usepackage{amsthm, amsfonts,  amsmath, amssymb}
\usepackage{url}
\usepackage[a4paper, top = 1.2in, bottom = 1.2in, left = 1.2in, right = 1.2in]{geometry}

\newtheorem{thm}{Theorem}[section]
\newtheorem{lem}[thm]{Lemma}
\newtheorem{prop}[thm]{Proposition}
\theoremstyle{definition}
\newtheorem{dfn}[thm]{Definition}
\newtheorem{Hyp}[thm]{Hypothesis}

\newtheorem{ex}{Example}

\newtheorem{remark}[thm]{Remark}
\theoremstyle{plain}
\newtheorem{cor}[thm]{Corollary}
\numberwithin{equation}{section}

\newcommand{\C}{\mathbb{C}}

\newcommand{\N}{\mathbb{N}}
\newcommand{\Q}{\mathbb{Q}}

\newcommand{\R}{\mathbb{R}}

\newcommand{\Z}{\mathbb{Z}}

\newcommand{\mcO}{\mathcal{O}}

\newcommand{\mfp}{\mathfrak{p}}
\newcommand{\mfq}{\mathfrak{q}}

\newcommand{\Gal}{\mathrm{Gal}}
\newcommand{\GL}{\mathrm{GL}}

\newcommand{\SL}{\mathrm{SL}}

\newcommand{\ra}{\rightarrow}

\newcommand{\mrm}[1]{\mathrm{#1}}

\newcommand{\pmat}[4]{ \begin{pmatrix} #1 & #2 \\ #3 & #4 \end{pmatrix}}

\newcommand{\psmat}[4]{\bigl( \begin{smallmatrix} #1 & #2 \\ #3 & #4 \end{smallmatrix} \bigr)}

\title[On the gaps between non-zero Fourier coefficients]{On the gaps between non-zero Fourier coefficients of cusp forms of higher weight}
\author[N. Kumar]{Narasimha Kumar}
\email{narasimha.kumar@iith.ac.in}
\address{
Department of Mathematics \\
Indian Institute of Technology Hyderabad\\
Kandi, Sangareddy - 502285\\
INDIA. 
}

\date{}
\begin{document}
\begin{abstract}
We show that if a modular cuspidal eigenform $f$ of weight $2k$ is $2$-adically close to an elliptic curve $E/\Q$, 
which has a cyclic rational $4$-isogeny, then $n$-th Fourier coefficient of $f$ is non-zero in the short interval $(X, X + cX^{\frac{1}{4}})$ 
for all $X \gg 0$ and for some $c > 0$. We use this fact to produce non-CM cuspidal eigenforms $f$ of level $N>1$ and weight $k > 2$ such that 
$i_f(n) \ll n^{\frac{1}{4}}$ for all $n \gg 0$.

\end{abstract}
  \maketitle

\section{Introduction}
The vanishing or non-vanishing of Fourier coefficients of cusp forms is one of the fundamental and interesting objects of study in number theory.  
In this article, we are interested in a question of Serre in bounding the maximum length of consecutive zeros of cusp forms of integral weight.  

In his paper~\cite{Ser81}, Serre proposes the study of the non-vanishing of Fourier coefficients of cusp forms in short-intervals. In particular, if
$f(z)= \sum_{n=1}^{\infty} a_f(n) q^n \in S_k(N)$, then he suggests the problem of finding upper bounds for the function $i_f(n)$ defined by
$$i_f(n) := \mrm{max}\ \{i: a_f(n+j)=0 \ \mrm{for\ all}\  0 \leq j \leq i\}.$$
In fact, he proved that if $f(z)$ is a cusp form of weight $k \geq 2$ which is not a linear combination of forms 
with complex multiplication (CM), then 
\begin{equation*}
i_f(n) \ll n. 
\end{equation*}
In same article, he poses a question, if the bound can be improved to $n^{\delta}$ with $0<\delta<1$. 

Through various approaches, many mathematicians have contributed in answering this question. The approaches are either using Rankin-Selberg estimates,
or Chebotarev density theorem, or distribution of $B$-free numbers, or the bounds on certain exponential sums, etc., 
(cf.~\cite{BO01},~\cite{Alk03},~\cite{Alk05},~\cite{AZ05}, \cite{AZ08},~\cite{DG14}). For example, the approach  to the non-vanishing problem through
the distribution of $B$-free numbers has been considered by Alkan and Zaharescu (\cite{AZ05JNT}), Matom\"aki (\cite{Mat12}), and Kowalski, Robert, and Wu (\cite{KRW07}),
and many more. For more details on the other approaches and the relevant literature on this problem, we urge the reader to look 
at the beautiful introductions of~\cite{BO01},~\cite{KRW07}. Currently, the best bound for $i_f(n)$ is available  due to Kowalski, Robert, and Wu 
and they proved that for any holomorphic non-CM cuspidal eigenform $f$ on general congruence groups, 
\begin{equation}
\label{IneqKRW}
i_f(n) \ll_{f} n^{7/17+\epsilon},
\end{equation}
holds for all $n \in \N$ (cf.~\cite{KRW07}). It is interesting to note that there is another aspect of understanding $i_f(n)$, 
that is, through the study of averages for $i_f(n)$ (cf.~\cite{AZ07}). 

In~\cite{AZ05IJNT}, for the first time, the authors have exploited  the idea of using congruences to study $i_f(n)$.
There it was shown that 
$i_{\Delta}(n) \ll n^{\frac{1}{4}}$, where $\Delta$ is the unique normalized cuspidal eigenform of weight $12$. 
The proof of this theorem relies on the existance of sums of squares in short intervals of the form $(x, x+x^{\frac{1}{4}})$.

In~\cite{DG14}, the authors extended the above idea on the existance of sums of squares in short intervals of the form $(x, x+cx^{\frac{1}{4}})$,
along with the congruences for eigenvalues of level $1$ Hecke eigenforms, to show that 
for any non-zero modular form $f \in S_{k}(\Gamma_0(1))$ with $k \geq 12$, one has 
\begin{equation}
\label{basicidentity}
i_f(n) \ll n^{1/4} \quad \forall \ n \gg 0,
\end{equation}
where the implied constant depends only on $k$. 

If the level $N>1$, then there are no similar general results are available with $i_f(n) \ll n^{\frac{1}{4}}$. However, in~\cite{DG15}, the authors were able to produce infinitely many 
non-isogenous elliptic curves for which~\eqref{basicidentity} holds.
In the same article, the authors have raised a question about extending the result to cuspidal eigenforms of higher weights. 
To our knowledge, there is not even a single example in the literature of a non-CM (or CM) cuspidal eigenform of level $N>1$ and weight $k>2$ 
for which the~\eqref{basicidentity} holds. 


The objective of this article to extend the results of~\cite{DG15} to elliptic curves with cyclic rational $4$-isogeny and
show that either if a modular cuspidal eigenform $f$ of weight $2k$ is $2$-adically close to an elliptic curve $E/\Q$, 
which has a cyclic rational $4$-isogeny or if there is a higher congruence for the prime $2$ holds between them,  then~\eqref{basicidentity} holds for $f$ as well. 
We use this fact to produce examples of non-CM, as well as CM, cuspidal eigenforms $f$ of level $N>1$ and weight $k > 2$ for which~\eqref{basicidentity} holds for $f$.  

\section{Acknowledgements}
The author would like to thank the referee for pointing out the references which are relevant to this problem and also for
his/her suggestions which improved the presentation of the article. We would also like to thank Dr. Satadal Ganguly 
for his suggestions during the preparation of this article. The basic idea of this article stems from the work of Das and Ganguly~\cite{DG14} 
and Rasmussen's thesis~\cite{Ras09} provided us a way to look for the examples that we need. This work was supported by IIT Hyderabad
through Institute's startup research grant.


\section{Elliptic curves with a cyclic rational $4$-isogeny}
Let $E$ be an elliptic curve over $\Q$ of conductor $N_E$. Let $f_E$ denote the cuspidal eigenform of weight $2$ with rational Fourier coefficients
corresponding to $E$, by the modularity theorem. Through out, we shall use $f_E$ to mean only this.

In~\cite{DG15}, it was shown that if the elliptic curve  $E$ has a rational $4$-torsion point, then for all $n \gg0$, $$i_{f_E}(n) \ll n^{\frac{1}{4}},$$  
where the implied constant depends on $N_E$. In this section, we extend this result to elliptic curves with a cyclic rational $4$-isogeny. 
Before we start stating our results, let us state useful results from ~\cite[Theorem 1]{DG14} and~\cite[Lemma 2.2]{KRW07}, which we shall use through out.
\begin{thm}[Das-Ganguly]
\label{final-step}
Given any integer $N \in \N$, there exists $X_0 \in \R^{+}$ and $c>0$ (depending only on $N$) such that 
there exists an integer $n$ which is a sum of two squares and co-prime to $N$ in intervals
of type $(X, X+cX^{\frac{1}{4}})$ for all $X \gg X_0$.
\end{thm}
\begin{lem}[Kowalski-Robert-Wu]
\label{krwlemma}
If $f = \sum_{n=1}^{\infty} a_f(n) q^n$ is a  normalized cuspidal eigenform in $S_{2k}(\Gamma_0(N),\chi)$, then there exists a natural number $M_f \geq 1$ such
that for any prime $p \nmid M_f$, either $a_f(p) = 0$ or $a_f(p^r) \neq 0$ for all $r \geq 1$. If $\chi$ is trivial and
$f$ has integer coefficients, then one can take $M_f = N$. 
\end{lem}

Given any elliptic curve $E$ over $\Q$ of conductor $N_E$, we let $\rho=\rho_E$ denote the Galois representation associated to $E$.
For any $N \in \N$, we let $\bar{\rho}_N$ denote the action of the Galois group $\Gal(\bar{\Q}/\Q)$ on $E[N]$, the $N$-torsion points of $E$.
\begin{thm}
\label{weight2theorem}
If an elliptic curve $E$ over $\Q$  of conductor $N_E$  has a cyclic rational $4$-isogeny, then  
$$ i_{f_E}(n) \ll n^{\frac{1}{4}} $$
for $n \gg 0$ and the implied constant depends only on $N_E$.
\end{thm}
\begin{proof}
The existence of a cyclic rational $4$-isogeny for $E$ implies that there exists an isogeny $\varphi: E \ra E^{\prime}$ defined over $\Q$ such that
$\mrm{ker}(\varphi)(\bar{\Q})$ is isomorphic to $\Z/4\Z$ and $\mrm{ker}(\varphi)$ is invariant under the $\Gal(\bar{\Q}/\Q)$-action. 
Therefore, the representation $\bar{\rho_4}$ is reducible and hence the image of $\bar{\rho_4}$ is inside the Borel subgroup of $\GL_2(\Z/4\Z)$,
and 
$$ \bar{\rho_4} \simeq \pmat{\chi}{*}{0}{*} \pmod 4. $$
For all primes $p \nmid 2N_E$, by calculating the trace and determinant of the representation at the Frobenius elements $\mrm{Frob}_p$ at $p$, 
we get the following congruence
$$a_E(p) \equiv \chi(p) + p \chi^{-1}(p) \pmod 4, $$
where $a_E(p)$ denotes the $p$-th Fourier coefficient of $f_E$. In fact, we obtain a relation that, for all $p \nmid 2N_E$ 
\begin{equation}
\label{twocases}
a_E(p) \equiv 1+p \pmod 4\ \mrm{or} \ a_E(p)\equiv -(1+p) \pmod 4 
\end{equation}
holds.

Now, we shall show that $a_E(n)$ is nonzero for all integers $n$ that are sums of two squares and are co-prime to $2N_E$.
If an integer $n$ is a sum of two squares then the prime factors of $n$ that are $\equiv 3 \pmod 4$ occur with an even exponent.
\begin{enumerate}
 \item If $p \equiv 1 \pmod 4$ and $p \nmid 2N_E$, then $a_E(p) \equiv 2 \pmod 4$ in either cases (cf.~\eqref{twocases}). 
       Therefore, $a_E(p)$ is non-zero if $p \equiv 1 \pmod 4$. 
       By Lemma~\ref{krwlemma}, we see that $a_E(p^m) \neq 0$ for $m \geq 1$, since $(p,N_E)=1$.
 \item If $p \equiv 3 \pmod 4$ and $p \nmid 2N_E$, then $a_E(p) \equiv 0 \pmod 4$.  The Hecke relations  
                         \begin{equation}
                         \label{EC Mod4}
                         a_E(p^r) = a_E(p)a_E(p^{r-1}) - p a_E(p^{r-2}),  
                         \end{equation}
       would imply that $a_E(p^{2r}) \equiv a_E(p^{2r-2}) \pmod 4$. Since $a_E(p^2) \equiv 1 \pmod 4$, we see that $a_E(p^{2r})$ is non-zero
       for all $r \geq 1$.
\end{enumerate}
By Theorem~\ref{final-step}, we can  obtain an integer  $n$ that is co-prime to $2N_E$ and is a sum of two squares in intervals
of type $(X, X+cX^{\frac{1}{4}})$, where $X \gg 0$ and $c>0$ depends only on $N_E$. Hence, we are done with the proof.
\end{proof}

Now, we shall show that Theorem~\ref{weight2theorem} holds for infinitely many non-isogenous elliptic curves $E$ over $\Q$.
To do this, let us recall some standard facts about modular curves.

If a field $k$ is of characteristic zero, then $k$-points of the curve $X_0(N)$ (away from the cusps) parametrize diagrams 
$(\varphi: E \ra E^{\prime})$ where $E, E^{\prime}$ are elliptic curves over $k$ and $\varphi: E \ra E^{\prime}$ is a $k$-rational isogeny with 
$\mrm{ker}(\varphi)$ over $\bar{k}$ is isomorphic to  a cyclic group of order $N$. 

For any prime $p$, let $X_0(p^2)$ be the classical modular curve over $\Q$ and $X^+_0(p^2)$ be its quotient
by the Atkin-Lehner involution. Let $X_{\mrm{split}}(p)$ be the modular curve defined over $\Q$ which corresponds
to the modular curve
$$\Gamma_{\mrm{split}}(p)= \left\{\psmat{a}{b}{c}{d} \in \mrm{SL}_2(\Z)|\ b \equiv c \equiv 0 \pmod p\ \mrm{or}\ a \equiv d \equiv 0 \pmod p \right\}, $$
i.e., $X_{\mrm{split}}(p) \otimes \C \simeq \Gamma_{\mrm{split}}(p)\backslash \mathbb{H} \cup \mathbb{P}^{1}(\Q)$.


\begin{prop}
\label{weight2infinite}
There exists infinitely many non-isogenous curves over $\Q$ which have a cyclic rational $4$-isogenies.
\end{prop}
\begin{proof}
For any prime $p$, the rational morphism of modular curves $X_0(p^2) \ra X_{\mathrm{split}}(p)$ 
induces an isomorphism of modular curves $X_0^+(p^2) \simeq X_{\mathrm{split}}(p)$(cf. the introduction of~\cite{BPR13}).  
If $p \leq 7$, we get that $X_{\mathrm{split}}(p)(\Q) \simeq \mathbb{P}^1(\Q)$(cf.~\cite[Page 115]{Mom84}).
When $p=2$, there are infinitely many non-isomorphic elliptic curves over $\Q$
which have a cyclic rational $4$-isogeny.
Now, the proposition follows from the fact that, given any elliptic curve $E$ over $\Q$, there are only finitely many elliptic curves over $\Q$
which are isogenous to $E$.
\end{proof}
\begin{remark}
In the proof of above proposition, one can also use the fact that $X_0(4)$ is a smooth algebraic curve of genus zero and hence is rational 
and that it has a rational point. Therefore, $X_0(4)$ has infinitely many rational points and hence so does $Y_0(4)$. 
On the contrary, in~\cite{Maz77}, Mazur showed that for each prime number $p =11$ or $p \geq 17$, 
there are only finitely many $\Q$-rational points on $X_{\mrm{split}}(p)$.
\end{remark}

In~\cite[Corollary 1]{DG15}, the authors have shown the existence of infinitely many non-isogenous elliptic curves over $\Q$ 
having a rational $4$-torsion point. As a consequence, they have produced infinitely many non-isogenous elliptic curves $E$ over $\Q$ with 
$i_{f_E} \ll n^{\frac{1}{4}}$ for all $n \gg 0$. We finish this section with a shorter proof of that result, 
by producing an explicit family of elliptic curves over $\Q$ having a rational $4$-torsion point.
\begin{prop}
For each $t \in \Q-  \{ 0, \frac{1}{4} \}$, the elliptic curve $E_t$ defined by 
$$E_t:y^2=x^3-(2t-1)x^2+t^2x$$
has a rational $4$-torsion point. In fact, there are infinitely many non-isogenous elliptic curves 
in the family $\{ E_t \}_{t \in \Q-  \{ 0, \frac{1}{4} \} }$.
\end{prop}
\begin{proof}
The point $(t,t)$ is a rational $4$-torsion point on $E_t$, therefore $E_t$ has a rational $4$-torsion point. 
Since the $j$-invariant of an elliptic curve  classifies elliptic curves over $\bar{\Q}$, up to an isomorphism,
we see that there are infinitely many elliptic curves in $\{ E_t \}$, which are not isomorphic to each other. 
This is because, given any element $\alpha \in \Q$, then $j(E_t)=\alpha$ holds only for finitely many $t$'s, since $j(E_t)$ is 
a rational polynomial in $\Q(t)$. Now, the last statement follows from the fact that, given any elliptic curve $E$ over $\Q$,
there are only finitely many elliptic curves which are isogenous to $E$.
\end{proof}
\begin{remark}
In fact, there is a two variable parametrization of elliptic curves over $\Q$ having $m$-torsion points, if $m \neq 1,2,3$. 
This can be found in the work of Lorenzini (cf.~\cite{Lor11} for more details).
\end{remark}

\section{Higher congruence between $f$ and $f_E$} 
In this section, we show that if there is higher congruence between $f$ and $f_E$ holds for the prime $2$, then~\eqref{basicidentity} holds for $f$ as well. 
In order to make these sense, let us define the notion of higher congruence between two eigenforms.

For $a \in \N$, a commutative ring $R$, and a formal power series
$$f = \sum_{n=0}^{\infty} c_n q^n \in R[[q]],$$
we define, for a prime ideal $\mfq$ of $R$,
$$\mrm{ord}_{\mfq^a} f = \mrm{inf}\{n \in \N \cup 0 \mid \mfq^a \nmid (c_n )\},$$ 
with the convention that $\mrm{ord}_{\mfq^a} f = \infty$ if $\mfq^a \mid (c_n)$\ for all $n$.
\begin{dfn}
We say that formal powers series $f_1$ and $f_2$ in $R[[q]]$ are congruent modulo $\mfq^a$,
if $\mrm{ord}_{\mfq^a}(f_1 - f_2) = \infty$, and we denote this by $f_1 \equiv f_2 \pmod {\mfq^a}.$
\end{dfn}

Let $f \in S_{2k}(\Gamma_0(N))$ be a cuspidal eigenform of weight $2k$ with $k>1$ and level $N$
with coefficient field $K$ and ring of integers $\mathcal{O}_K$. Let $N_f := \mrm{lcm}(N,M_f)$,
where $M_f$ is a natural number corresponding to $f$ as in Lemma~\ref{krwlemma}. Let $\mfq$ be a prime ideal of
$\mathcal{O}_K$ lying above $2$ and let $e(\mfq/2)$ denote the ramification index of $\mfq$.

Let $f_E$ be the cuspidal eigenform of level $N_E$ corresponding to an elliptic curve $E$ over $\Q$ of conductor $N_E$, which has a cyclic rational $4$-isogeny.

\begin{thm}
\label{maintheorem}
Let $f$ and $f_E$ be as above. If $f \equiv f_E \pmod {\mfq^m}$  for some $m >e(\mfq/2)$, 
then 
\begin{equation}
\label{localthmidentity}
i_f(n) \ll n^{\frac{1}{4}}, 
\end{equation}
for all $n \gg 0$ and the implied constant depends only on $N_fN_E$.
\end{thm}


\begin{proof}
Since $f \equiv f_E \pmod {\mfq^m}$, we see that $$a_f(n) \equiv a_{f_E}(n) \pmod {\mfq^m},$$ for all $n \geq 1$. 
By the modularity theorem, we know $a_{f_E}(n)=a_E(n)$ for all $n \geq 1$. Hence, $$ a_f(n) \equiv a_E(n) \pmod {\mfq^m}.$$

We show that $a_f(n)$ is non-zero for all integers $n$ that are sums of two squares and co-prime to $2N_fN_E$.
Observe that, if an odd integer $n$ is a sum of two squares then the prime factors of $n$ that are $\equiv 3 \pmod 4$ occur with an even exponent.
\begin{enumerate}
 \item If $p \equiv 1 \pmod 4$ and $p \nmid 2N_fN_E$, then $a_f(p) \equiv \pm (p+1) \pmod {\mfq^{\mrm{min}\{m,2e(\mfq/2)\}}}$, by~\eqref{twocases}.
       This implies that, the $a_f(p)$ is non-zero, because $m > e(\mfq/2)$. 
       By Lemma~\ref{krwlemma}, we see $a_f(p^m) \neq 0$ for $m \geq 1$, since $(p,N_f)=1$.
 \item If $p \equiv 3 \pmod 4$ and $p \nmid 2N_fN_E$, then $a_f(p) \equiv \pm (p+1) \pmod {\mfq^m}$, by~\eqref{twocases}.  
       \begin{itemize}
        \item If $a_f(p) \equiv 0 \pmod {\mfq^m}$, then by Hecke relations of Fourier coefficients of $f$, we see that 
       $$a_f(p^{2n}) \equiv (-p)^{n} \pmod {\mfq^m},$$ hence $a_f(p^{2n})$ is non-zero for all $n \geq 1$,
       because $p$ and norm of $\mfq$ are relatively prime.
       \item  If $a_f(p) \not \equiv 0 \pmod {\mfq^m}$, then 
       $a_f(p)$ is non-zero, and by Lemma~\ref{krwlemma}, we see $a_f(p^{2n}) \neq 0$ for $n \geq 1$, since $p \nmid N_f$.

       \end{itemize}

\end{enumerate}
By Theorem~\ref{final-step}, we can  obtain an integer  $n$ that is co-prime to $2N_fN_E$ and is a sum of two squares in intervals
of type $(X, X+cX^{\frac{1}{4}})$, where $X \gg 0$ and $c>0$ depends only on $N_fN_E$. Hence, we are done with the proof.
\end{proof}
As mentioned in the introduction, by using Theorem~\ref{maintheorem}, we wish to produce examples of cuspidal eigenform of level $N>1$ and weight $k>2$.
A special case of the above theorem is useful in producing examples, because in this case we explicitly know the integer $M_f$ and hence $N_f$.
\begin{cor}
Let $f$ be as in the theorem above and assume that the Fourier coefficients of $f$ are integers. 
If $f \equiv f_E \pmod 4$, then $$i_f(n) \ll n^{\frac{1}{4}},$$
for all $n \gg 0$ and the implied constant depends only on $NN_E$.
\end{cor}

In order to produce examples using the above corollary, one may need to check congruences for infinitely many coefficients.
In the case of mod $p$ congruences, this problem reduces to checking the congruences for only finitely many coefficients,
due to the work of Sturm. In our situation, to check the higher congruences between modular forms, one requires similar results. 
Thanks the work of Rasmussen, we do have similar results with appropriate bounds, which we refer to them as Sturm's bound.

The following proposition is a particular case of~\cite[Prop. 2.8]{Ras09} and we state his result only when $p=2$ with $m=2e(\mfq/2)$ and when $12$ divides $N$.
\begin{prop}
\label{Sturm_general}
Let $f_1$ and $f_2$ be modular forms of weights $2k_1$ and $2k_2$ of level $N$ and with coefficients in $\mcO_K$.
Let $\mfq$ be a prime ideal of $\mcO_K$ above $2$.
If $a_n(f_1) \equiv a_n(f_2) \pmod {\mfq^m}$ for all $n \leq B:=\mrm{max}\{2k_1,2k_2\} [\SL_2(\Z):\Gamma_1(N)]/12$, we have
$$f_1 \equiv f_2 \pmod {\mfq^m}.$$
\end{prop}
\begin{remark}
We refer to the number $B$ in the above proposition as the Sturm's bound. It only depends on the level and weights of 
the modular forms involved, but  not on the dimension of the space where the modular forms belong to.	
\end{remark}

In the following examples, we apply Proposition~\ref{Sturm_general} with $m=2$, and $\mfq=(2)$. Clearly, we have $e(\mfq/2)=1 < m=2 $.
Recall that, the Cremona label of an elliptic curve $E$ over $\Q$ is a way of indexing the elliptic curves over $\Q$. 
The first number represents the conductor of $E$, the letter(s) followed by represent the isogeny class of $E$ and 
the last number represents the isomorphism class within the isogeny class of $E$ as it appears in Cremona's tables.
\begin{ex}
Let $E$ be the elliptic curves over $\Q$ defined by $y^2=x^3-x^2-64x+220$ of conductor $24$ and  has  a cyclic rational $4$-isogeny. 
Cremona label for $E$ is $24a3$. Let $f_E$ denote the cuspidal non-CM eigenform of level $24$ corresponding to $E$. 
The dimension of $S^{\mrm{new}}_2(24)$ is $1$ and the $q$-expansion of $f_E$ is given by
\begin{equation}
\label{wt2}
f_E= q - q^3 -2q^5 +q^9 +4q^{11} - 2q^{13} +2q^{15} +2q^{17} - 4q^{19}  + O(q^{20}). 
\end{equation}
\end{ex}

\begin{ex}
The dimension of $S^{\mrm{new}}_4(24)$ is $1$ and generated by $f_{24,4}$. The $q$-expansion of $f_{24,4}$ is given by 
\begin{equation}
\label{wt4} 
f_{24,4} = q + 3q^3 + 14q^5 - 24q^7 + 9q^9 - 28q^{11} -74q^{13} +42 q^{15} +82 q^{17} + 92q^{19} + O(q^{20}).
\end{equation}
\end{ex}

\begin{ex}
The dimension of the $S^{\mrm{new}}_{10}(12)$ is $1$ and generated by $f_{12,10}$.  The $q$-expansion of $f_{12,10}$ is given by
\begin{equation}
\label{wt10}
f_{12,10}= q -  81 q^3 + 990 q^5 + 8576 q^7 + 6561 q^9 + 70596 q^{11} -   2530 q^{13} + O(q^{14})
\end{equation}
\end{ex}



\begin{thm}
The cuspidal eigenforms $f$ in~\eqref{wt2},~\eqref{wt4},~\eqref{wt10}  
are without complex multiplication and satisfy 
\begin{equation}
\label{localidentity}
i_f(n) \ll n^{\frac{1}{4}}  
\end{equation}
for $n \gg 0$ and the implied constant depends on levels.
\end{thm}
\begin{proof}
From the tables in~\cite{Tsa14}, we see that all the cuspidal Hecke eigenforms in~\eqref{wt2},~\eqref{wt4},~\eqref{wt10} are without
complex multiplication. By Proposition~\ref{Sturm_general}, it is enough to check the congruences between the eigenforms in~\eqref{wt2},~\eqref{wt4},~\eqref{wt10}
for all coefficients up to the corresponding Sturm's bound.
\begin{enumerate}
 \item By Theorem~\ref{weight2theorem}, the cuspidal eigenform $f_E$ satisfies~\eqref{localidentity}.
 \item For the eigenform $f_{24,4}$, if we show that $f_E$ and $f_{24,4}$ are congruent modulo $4$, then the claim follows from Theorem~\ref{maintheorem}.
       By Proposition~\ref{Sturm_general}, it is enough to check these congruence for all coefficients up to $128$.  
       Using SAGE, we have checked that these indeed hold and hence the cuspidal eigenform $f_{24,4}$ satisfies~\eqref{localidentity}.
 \item For the eigenform $f_{12,10}$, if we show that $f_E$ and $f_{12,10}$ are congruent modulo $4$, then the claim follows from Theorem~\ref{maintheorem}.
       To check these congruences between $f_{12,10}$ and $f_E$, we treat the cuspidal eigenform $f_{12,10}$ as a cusp form in $S_{10}(24)$ 
       and apply the Proposition~\ref{Sturm_general}. Now, by Proposition~\ref{Sturm_general}, it is enough to check the congruence for all coefficients up to $320$.  
       Using SAGE, we have checked that these indeed hold  and hence the cuspidal eigenform $f_{12,10}$ satisfies~\eqref{localidentity}.     
\end{enumerate}
\end{proof}

We remark that Theorem~\ref{maintheorem} works for any cuspidal Hecke eigenform $f$, i.e.,
there is no assumption on the Hecke eigenform $f$ being CM or non-CM.
However, the results of Balog and Ono (\cite{BO01}), Kowalski, Robert, and Wu (\cite{KRW07}) were proved for Hecke eigenforms 
without CM. So, we chose to construct  cuspidal Hecke eigenforms of higher weight without CM to show that there are 
examples with an improved bound for $i_f(n)$. Also, note that all the eigenforms in~\cite{DG14} are also without CM, because
they are of level $1$.

Now we shall  give an example of a Hecke eigenform $f$ of weight $k>2$ and level $N>1$ with CM and 
satisfy the condition that $i_f(n) \ll n^{\frac{1}{4}}$.

\begin{ex}
Let $E$ be the elliptic curves over $\Q$ defined by $y^2=x^3-11x+14$ of conductor $32$ and  has  a cyclic rational $4$-isogeny. 
Cremona label for $E$ is $32a4$. Let $f_E$ denote the cuspidal CM eigenform of level $32$ corresponding to $E$. 
The dimension of $S^{\mrm{new}}_2(32)$ is $1$ and the $q$-expansion of $f_E$ is given by
\begin{equation}
\label{wt2cm}
f_E= q-2q^5-3q^9+6q^{13}+2q^{17}+ O(q^{20}). 
\end{equation}
\end{ex}

\begin{ex}
\label{wt4cm} 
The dimension of $S^{\mrm{new}}_4(32)$ is $3$ and denote the basis of eigenforms with $f^{(1)}_{32,4},f^{(2)}_{32,4},f^{(3)}_{32,4}$, respectively. 
For example, the $q$-expansion of $f^{(3)}_{32,4}$ is given by
\begin{equation}
f^{(3)}_{32,4} = q + 8q^3 -10q^5 +16q^7 + 37q^9 - 40q^{11} -50q^{13} -80 q^{15} -30 q^{17} + 40q^{19} + O(q^{20}).
\end{equation}
\end{ex}

\begin{thm}
The cuspidal eigenform $f_E$ in~\eqref{wt2cm} and one of the $3$ eigenforms in Example~\ref{wt4cm} 
are with CM and satisfy 
\begin{equation}
\label{localidentitycm}
i_f(n) \ll n^{\frac{1}{4}}  
\end{equation}
for $n \gg 0$ and the implied constant depends on the levels.
\end{thm}
\begin{proof}
From the tables in~\cite{Tsa14}, we see that the cuspidal Hecke eigenform $f_E$ in~\eqref{wt2cm} has CM.
By Theorem~\ref{weight2theorem}, the cuspidal eigenform $f_E$ satisfies~\eqref{localidentitycm}.

Interestingly, all the $3$ cuspidal eigenforms in  $S_{4}^{\mrm{new}}(32)$ and the elliptic curve $f_E$ are congruent modulo $4$. 
By Proposition~\ref{Sturm_general}, it is enough to check these congruences up to the corresponding Sturm's bound, which is $256$ in this case. 
Using SAGE, we have checked that these indeed hold and hence all these $3$ cuspidal eigenforms in  $S_{4}^{\mrm{new}}(32)$ satisfies 
the identity~\eqref{localidentitycm}, by Theorem~\ref{maintheorem}.
From the tables in~\cite{Tsa14}, we see that, out of these $3$ eigenforms $f^{(1)}_{32,4},f^{(2)}_{32,4},f^{(3)}_{32,4}$, 
$2$ are without CM and $1$ is with CM. 
\end{proof}
Hence, there is an example of a Hecke eigenform $f$ of weight $k>2$ and level $N>1$ with CM and 
satisfy the condition that $i_f(n) \ll n^{\frac{1}{4}}$.



\section{$f$ and $f_E$ are $2$-adically close}
In this section, we improvise the result in the previous section and prove that if a modular cuspidal eigenform $f$ is $2$-adically close to an elliptic curve 
$E/\Q$, which has a cyclic rational $4$-isogeny, then~\eqref{basicidentity} holds for $f$ as well.  We also produce an example where this improvised version of the result is needed.

Define a function $\alpha:\Z \ra \Z$ as follows:
\begin{equation*}
\alpha(n) = \left\{
      \begin{array}{rl}
          0 & \text{if } n \leq 1,\\
           1 & \text{if } n = 2,\\
           n-2 & \text{if } n > 2.
\end{array} \right.
\end{equation*}
Now, we shall introduce the notion of $2$-adically close.
\begin{dfn}
Let $k_1, k_2$ be positive integers such that $2k_1 \equiv 2k_2 \pmod {2^s}$ for some integer $s \geq 1$. For $i=1,2$, suppose $f_i$ is a cuspidal eigenform 
on $\Gamma_0(N_i)$ of level $N_i$ and weight $k_i$ with coefficients in $\mcO_K$.

We say that $f_1$ and $f_2$ are $2$-adically close, if there exists a prime ideal $\mfq$ over $2$ in $\mcO_K$ with ramification index $e(\mfq/2)$
and an integer $m$ with $s \geq \alpha(\lceil \frac{m}{e(\mfp)/2} \rceil) \geq 1 $ such that
\begin{equation*}
a_{f_1}(p) \equiv a_{f_2}(p) \pmod {\mfq^m},
\end{equation*}
for all primes $p \nmid 2N_1N_2$.
\end{dfn}
\begin{remark}
\label{2adicremark}
The condition $\alpha\left(\lceil m/e(\mfq/2) \rceil \right) \geq 1$ implies that $m > e(\mfq/2)$.
\end{remark}

Let $f \in S_{2k}(\Gamma_0(N))$ be a cuspidal eigenform of level $N>1$ and weight $2k$ with $k>1$ with coefficients in the ring of 
integers $\mathcal{O}_K$. Let $N_f := \mrm{lcm}(N,M_f)$,
where $M_f$ is a natural number corresponding to $f$ as in Lemma~\ref{krwlemma}.
Let $\mfq$ be a prime ideal of
$\mathcal{O}_K$ lying above $2$ and let $e(\mfq/2)$ denote the ramification index of $\mfq$. 

Let $f_E$ be a cuspidal eigenform of level $N_E$ corresponding to an elliptic curve $E$ over $\Q$ of conductor $N_E$, which has a cyclic rational $4$-isogeny.


\begin{thm}
\label{maintheorem-2}
Let $f$ and $f_E$ be as above.  If $f$ and $f_E$ are $2$-adically close, then 
\begin{equation}
\label{localrefinedidentity}
i_f(n) \ll n^{\frac{1}{4}}, 
\end{equation}
for all $n \gg 0$ and the implied constant depends only on $N_fN_E$.
\end{thm}
\begin{proof}
The proof of this theorem is essentially same as the proof of Theorem~\ref{maintheorem}, by Remark~\ref{2adicremark}.
This is because, for some $m >e(\mfq/2)$, we have $$a_f(p) \equiv a_{f_E}(p) \pmod  {\mfq^m}$$
for all $p \nmid 2N_fN_E$. By~\cite[Prop. 2.10]{Ras09}, this congruence can be extended to 
$$a_f(p^n) \equiv a_{f_E}(p^n) \pmod  {\mfq^m} \quad \forall \ n \geq 1,$$
since $s \geq \alpha(\lceil \frac{m}{e(\mfp)/2} \rceil)$.
By the multiplicativity properties of Fourier coefficients, we get that 
$$a_{f}(n) \equiv a_{f_E}(n) \pmod {\mfq^m}$$ 
for all $(n,2N_fN_E)=1$. Towards the end of the proof,  we are again interested only in sums of squares $n$ in short intervals $(X,X+cX^{\frac{1}{4}})$ 
which are relatively prime to $2N_fN_E$. Hence, we are done with the proof.
\end{proof}
A special case of the above theorem is useful in producing examples of higher weight and level, because in this case we explicitly know the integer $M_f$,
and hence $N_f$.
\begin{cor}
Let $f$ be as in the theorem above and  assume the Fourier coefficients of $f$ are integers. 
If $f \equiv f_E \pmod 4$, then $$i_f(n) \ll n^{\frac{1}{4}},$$
for all $n \gg 0$ and the implied constant depends only on $NN_E$.
\end{cor}

Now, we shall produce an example of Theorem~\ref{maintheorem-2}, where we take $s=1$, $m=2$, and $\mfq=(2)$.

Let $E$ be the elliptic curves over $\Q$ defined by $y^2+xy+y=x^3+x^2-80x+242$ of the conductor $15$ and has  a cyclic rational $4$-isogeny. 
Cremona label $E$ is $15a7$. Let $f_E$ denote the cuspidal eigenform of level $15$ corresponding to $E$.
The dimension of $S^{\mrm{new}}_2(15)$ is $1$-dimensional and  the $q$-expansion $f_E$ is given by
$$f_E = q - q^2 - q^3 - q^4 + q^5 + q^6 + 3q^8 + q^9 - q^{10} - 4q^{11} + O(q^{12}).$$

The dimension of $S_4^{\mrm{new}}(15)$ is $2$-dimensional, but we take the form with the following $q$-expansion
$$f_{15,4} = q + q^2 + 3q^3 - 7q^4 + 5q^5 + 3q^6 - 24q^7 - 15q^8 + 9q^9 + 5q^{10} + 52q^{11} + O(q^{12}).$$

Observe that in the above  eigenforms, the coefficients of $q^2$ in $f_E$ and $f_{15,4}$ are not congruent mod $4$, 
so one cannot use Theorem~\ref{maintheorem} to prove~\eqref{localrefinedidentity}. However, we can use Theorem~\ref{maintheorem-2} to prove it.

From the tables in~\cite{Tsa14}, we see that the cuspidal  eigenforms $f_E,f_{15,4}$ are without CM.
In order to check whether $f_E$ and $f_{15,4}$ are $2$-adically close, it is enough to check 
the congruences for all primes $p \neq 2,3,5$ and up to a certain Sturm bound $B$ by~\cite[Corollary 2.14]{Ras09}
or by ~\cite[Part (ii) of Thm 1]{CKR10}. This verification was done in the Appendix A of Rasmussen's Thesis~\cite[Page 67]{Ras09}.


\section{Theorem~\ref{maintheorem} for an integral cusp form $f$}
%

In this section, we would like to see how far Theorem~\ref{maintheorem} is true when the cuspidal eigenform $f$ is replaced by
a cusp form with integral Fourier coefficients and is congruent to $f_E$ modulo $4$. 
Since we are not working any more with an eigenform, one cannot expect that the result would be as strong as in the Theorem~\ref{maintheorem}.
Though, the results are weaker, still one can say some thing about $i_f(n)$. For our convenience, we state the results only when $f$ has 
integer Fourier coefficients, but the same arguments work even when $f$ has integral Fourier coefficients but by modifying the hypothesis
appropriately.

To our knowledge, all the existing results are known only when $f$ is a cuspidal eigenform without CM, but not for any integral cusp form $f$,
except when the level is $1$. In order to state the results of this section, we need the following Hypothesis.

\begin{Hyp} 
\label{Hypo_Key}
Given an $N\in\N$, there exists $\delta = \delta(N) \in \R^{+}$ and $c=c(N) \in \R^{+}$
with the following property: For $X \gg 0$, every short interval $(X, X+cX^{\delta})$ contains an integer $n$ with $(n,N)=1$
and  of the form $ 2^ip^jm^2$, where $p \equiv 1 \pmod 4$ with $j \not \equiv 3 \pmod 4$ and $(p,m)=1$.
\end{Hyp}

\begin{thm}
\label{secondmaintheorem}
Let $E$ denote an elliptic curve over $\Q$ of conductor $N_E$. Assume that $E$ has a multiplicative reduction at $2$ and a cyclic rational $4$-isogeny. 
If $f \in S_{2k}(N)$ is a cusp form, with integer Fourier coefficients, is congruent to 
$f_E$ modulo $4$ and the Hypothesis~\ref{Hypo_Key} holds for $NN_E$, then 
\begin{equation}
\label{basicrefinedidentity}
   i_f(n) \ll n^{\delta}, 
\end{equation}
 for $n \gg 0$ and the implied constant depends only on $NN_E$. 
\end{thm}
\begin{remark}
Since $f$ is only a cusp form, we don't have any information about the Hecke  relations of the Fourier coefficients of $f$,
but by the hypothesis we know that $f \pmod 4$ satisfies the Hecke  relations.
\end{remark}

\begin{proof}
Since $f \equiv f_E \pmod 4$, we see that $a_f(n) \equiv a_{f_E}(n) \pmod 4$, for all $n \geq 1$. By the modularity theorem, we know that $a_{f_E}(n)=a_E(n)$
for all $n \geq 1$. Hence, $$ a_f(n) \equiv a_E(n) \pmod 4\ \mrm{for \ all}\ n \geq 1.$$
We show that $a_f(n) \not \equiv 0 \pmod 4$ for all integers $n$ as in Hypothesis~\ref{Hypo_Key}.
\begin{enumerate}
 \item If $p=2$, then $a_f(2) \equiv \pm 1 \pmod 4$, since $E$ has a multiplicative reduction at the prime $2$. 
       Therefore, $a_f(2^n) \equiv a_f(2)^n \pmod 2$, hence $a_f(2^n)$ is a $2$-adic unit. In particular, for all $n \geq 1$  we have
       \begin{equation}
         \label{congruencemod2}
          a_f(2^n) \equiv \pm 1 \pmod 4.
       \end{equation}

 \item If $p \equiv 3 \pmod 4$ and $p \nmid NN_E$, then $a_f(p) \equiv 0 \pmod 4$. The mod $4$-Hecke relations among the Fourier coefficients for $E$ 
       and hence for $f$ implies that 
       \begin{equation}
          \label{congruencepmod3}
       a_f(p^{2n}) \equiv 1 \pmod 4
       \end{equation}
       for all $n \geq 1$. 
       
 \item If $p \equiv 1 \pmod 4$ and $p \nmid NN_E$, then $a_f(p) \equiv 2 \pmod 4$. 
       By using the mod-$4$ Hecke  relations of $f$, we get $$a_f(p^{m+4}) \equiv a_f(p^m) \pmod 4.$$ One can see that $a_f(p) \equiv 2 \pmod 4$, $a_f(p^3) \equiv 0 \pmod 4$ and
       \begin{equation}
          \label{congruencepmod1}
       a_f(p^2) \equiv 3 \pmod 4, a_f(p^4) \equiv 1 \pmod 4. 
       \end{equation}
 \end{enumerate}
For any $n =2^ip^jm^2$ where $p \equiv 1 \pmod 4$ with $j \not \equiv 3 \pmod 4$, $(p,m)=1$ and $(n,NN_E)=1$, by~\eqref{congruencemod2},\eqref{congruencepmod3},\eqref{congruencepmod1}, we have
$$ a_f(n) \equiv a_f(2^i) a_f(p^j) a_f(m^2) \equiv \pm 1, \pm 2 \pmod 4.$$
Hence, $a_f(n)$ is non-zero.  By Hypothesis~\ref{Hypo_Key}, we can find such integers $n=2^ip^jm^2$ in short intervals $(X,X+cX^{\delta})$,
hence we are done with the proof.
  \end{proof}
\begin{remark}
The reason for having Hypothesis~\ref{Hypo_Key} in the statement of the above theorem is that in part $3$ of the above proof, 
due to the lack of Hecke relations of Fourier coefficients of $f$, we don't get to know the non-vanishing of $a_f(p^{n})$
if $n \equiv 3 \pmod 4$, and of $a_f(p_1p_2)$ if both primes $p_1,p_2$ are $\equiv 1 \pmod 4$.
\end{remark}
   
In the previous section, we have produced examples for which~\eqref{localrefinedidentity} holds. However, we are not aware of a way to produce infinitely many examples.
However, using Theorem~\ref{secondmaintheorem}, one can produce infinitely many cusp forms with $\delta=\delta(N)>0$ if Hypothesis~\ref{Hypo_Key} holds for some $N$ for which
there exists an elliptic curve of conductor $N$, which has a cyclic rational $4$-isogeny. Let us explain what we mean with the following proposition.

Take $E_0: y^2 + xy + y = x^3 + x^2 - 1344x + 18405$ whose conductor is $42$ and its torsion subgroup is isomorphic to $\Z/4\Z$. 
There is no specific reason for this choosing $E_0$, any elliptic curve with the same property will also work.
\begin{prop}
Assume that Hypothesis~\ref{Hypo_Key} holds for $N_{E_0}=42$ with some $\delta>0$. Then, there exists infinitely many cusp forms $f$ in $S_k(\Gamma_0(N))$, as $k, N$ varies with $k>2, N > 2$ such 
that $$ i_f(n) \ll n^{\delta},$$
for $n \gg 0$ (depends only on $N_{E_0}$).
\end{prop}
\begin{proof}
Let $E_4$ denote the Eisenstein series of level $1$, weight $4$. Since $E_4 \equiv 1 \pmod 4$, we see that 
$$a_{f_{E_0}E_4^n}(m) \equiv a_{f_{E_0}}(m) \pmod 4$$
for all $m \in \N$. By Theorem~\ref{secondmaintheorem}, we have that~\eqref{basicrefinedidentity} holds for $f_{E_0}E_4^n$ as well.
Since the weights of $f_{E_0} E^n_4 \in S_{4n+2}(84)$ are distinct, these cusp forms cannot be the same. Hence, we are done.
\end{proof}
Finally, we end this article with a possibility of getting rid of the choice of $\delta$ through out, but the author does not know how to implement this idea.
\begin{remark}
For every elliptic curve $E$ over $\Q$ with a cyclic rational $4$-isogeny, if there is a lift of the mod $4$ eigenform $f_{E} E^n_4 \pmod 4$ to an eigenform 
$g_n$ of level $N$ and weight $4n+2$, for some $n \in \N$, then $g_n$ satisfies~\eqref{localrefinedidentity} (by Theorem~\ref{maintheorem-2}). The choice
$n$ can depend on the elliptic curve $E$. Then, by Proposition~\ref{weight2infinite}, we can produce infinitely many examples of cuspidal eigenforms weight $k > 1$ and $N>1$ 
for which~\eqref{localrefinedidentity} holds

In case, if the above statement is true for all $n\in \N$ for some elliptic curve $E$ over $\Q$ with a cyclic rational $4$-isogeny, then also we can produce infinitely many examples of cuspidal eigenforms weight $k > 1$ and $N>1$ 
for which~\eqref{localrefinedidentity} holds. 
\end{remark}

\bibliographystyle{plain, abbrv}

\end{document}